\documentclass{amsart}

\usepackage{amsmath,amssymb,mathrsfs}
\usepackage{stackengine,scalerel}

\newcommand{\abs}[2][\empty]{\ifx#1\empty\left|#2\right|%
\else#1\vert #2 #1\vert\fi}
\newcommand{\approxin}{\mathrel{\mathop{\ooalign{\kern.15ex\raisebox{.1ex}{\scalebox{.85}{$\sim$}}\cr$\subset$\cr}}}}
\newcommand{\Cnt}[1][]{{\mathcal C}^{#1}}
\newcommand{\conv}{\star}
\newcommand{\Fin}{\mathop\mathrm{Fin}\nolimits}
\newcommand{\fourier}[1]{\widehat{#1}}
\DeclareMathOperator{\st}{st}
\newcommand{\test}{\Cnt[\infty]_c}
\stackMath
\def\ster#1{\leavevmode\ThisStyle{%
\def\stackalignment{r}\def\stacktype{L}%
  \mkern-4mu\stackon[0pt]{\SavedStyle\phantom{f}#1}  
    {\SavedStyle^*\mkern-1.7mu\phantom{#1}}%
}}

\def\N{\mathbb  N}
\def\Q{\mathbb  Q}
\def\R{\mathbb  R}

\theoremstyle{plain}
\newtheorem{theorem}{Theorem}
\newtheorem{corollary}[theorem]{Corollary}
\newtheorem{definition}[theorem]{Definition}
\newtheorem{lemma}[theorem]{Lemma}

\theoremstyle{definition}
\newtheorem{example}[theorem]{Example}
\newtheorem{remark}[theorem]{Remark}

\begin{document}
\title{Generalized functions with infinitesimals}
\author[H.\ Vernaeve]{Hans Vernaeve}
\address{Department of Mathematics: Analysis, Logic and Discrete mathematics, Ghent University, Krijgslaan 297, B 9000 Gent, Belgium}
\email{hans.vernaeve@ugent.be}

\keywords{distribution theory, Schwartz distributions, generalized functions, nonstandard analysis}

\subjclass[2010]{26E35, 46F05, 46S20}

\begin{abstract}
We give a survey of the use of infinitesimals within mathematical analysis to rigorously deal with the delta-function from physics, and more generally, with distributions in the sense of L.\ Schwartz. We use the framework of nonstandard analysis as introduced by A.\ Robinson to rigorously deal with infinitesimals. Our exposition tries to be elementary, except for familiarity with nonstandard analysis, and takes as a starting point one of the basic questions in PDE for which distribution theory was developed. No knowledge of distribution theory is required.
\end{abstract}

\maketitle

\begin{quote}
To Chris Impens, who knew all along that distributions can be more intuitively described with infinitesimals.
\end{quote}

\section{Introduction}
Starting as early as 1822 with the work of J.\ Fourier, infinitesimal and infinitely large quantities have been used in mathematical analysis to define functions with delta function behaviour \cite[\S 5.1]{Laugwitz89}. One century later, infinitesimal quantities have been banned from analysis because it was unknown at the time how to use them rigorously. At the same time, the physicist P.\ Dirac coined the notion of delta function as a tool in quantum mechanics \cite{Dirac}.
Mathematicians like S.\ Sobolev and L.\ Schwartz subsequently introduced ways to rigorously define the delta function and weak solutions for partial differential equations (PDE). (Both concepts are related through the concept of a fundamental solution.) Their method is more abstract, defining the delta function and fundamental solutions of PDE as continuous linear functionals on suitable function spaces \cite{Schwartz}.\\
It can be safely assumed that Schwartz was not led to his definitions by an intuition with infinitesimals, but rather by an intellectually demanding process of redefining nonrigorously defined objects by extracting enough abstract properties to describe them, which fits the style of the Bourbaki group to which he belonged. As Schwartz describes in his biography \cite[pp.\ 218, 230, 232]{Schwartz01}:
\begin{quote}
\emph{[T]hose formulas [involving the Dirac $\delta$ function] were so crazy from the mathematical point of view that there was simply no question of accepting them. It didn't even seem possible to conceive of a justification.} (...)\\
\emph{I quickly understood that I had come upon everything I had been searching for for more than ten years.} (...) \emph{Functions were operators, but there were many operators which were not functions, such as $\delta$ and $\delta'$. The mystery of the Dirac $\delta$ function and its derivatives was solved.} (...) \emph{I always called the night of my discovery [of distributions] a marvelous night, the most beautiful night of my life.}
\end{quote}
The mathematical community completely accepted this theory and considered it a significant achievement, for which Schwartz was awarded a Fields Medal in 1950:
\begin{quote}
\emph{Developed the theory of distributions, a new notion of generalized function motivated by the Dirac delta-function of theoretical physics.}\\
(The IMU awarding the Fields Medals, about L.\ Schwartz)
\end{quote}

Finally, in the early 1960s, after centuries of controversy, A.\ Robinson found a way to rigorously deal with infinitesimals and infinitely large quantities, which he called \emph{nonstandard analysis}. Already in his 1966 book \cite{Robinson}, Robinson included a section on Schwartz distributions, representing them by ordinary functions defined on $\ster\R$, the extension of the real line $\R$ with infinitesimal and infinitely large quantities. The idea is quite intuitive and close to the historical development, and avoids the relatively heavy functional analysis required in Schwartz's theory, which makes a modern first course in distribution theory spend a large amount of time introducing---often without much motivation---the necessary function spaces, their duals, and their topological properties, and by extending a number of familiar operations on functions to distributions. It led historian D.\ Laugwitz to write in retrospect about Schwartz's theory \cite[p.169]{Laugwitz86}:
\begin{quote}
\emph{Die Distributionen erscheinen hier} (\dots) 
\emph{eher als Ersatzkonstruktionen f\"ur das Infinitesimale.}\\
(The distributions appear here rather as replacement constructions for infinitesimals.)
\end{quote}

Since then, several texts on distributions with infinitesimals have been published. On one hand, there are a number of elementary, often didactical texts representing the delta function as a function defined on $\ster\R$, but hardly developing any theory with them. On the other hand, there are a number of texts which show that any distribution can be represented as a nonstandard function, and leave the significance of this result to the reader's knowledge of distribution theory \cite{Robinson},\cite[Ch.~6]{Hoskins-Pinto}. In \cite{SL}, a significant body of abstract functional analysis has been treated within nonstandard analysis, and Schwartz's approach to distribution theory can be considered a special case of this theory (\cite[\S 10.4]{SL} calls their own account of distribution theory a sketch). This approach puts a heavy burden on the reader.

The current exposition intends to fill this gap by showing from scratch how nonstandard functions can be used to solve problems in the spirit of distribution theory, gradually discovering a function space which, after a suitable identification, turns out to be isomorphic to the space of distributions, even though we don't need the latter fact to develop the theory. In retrospect, one could imagine that this approach would have been the natural approach, had Robinson's discovery been made half a century earlier. The way in which distributions are represented in nonstandard analysis sometimes differs between different texts, but e.g., the approaches in \cite{Hoskins-Pinto} and \cite{SL} turn out to be equivalent \cite{Vernaeve}. Our definitions will be close to those in \cite{SL}.

Finally, we point out that distributions may be defined in several ways, but that the resulting theories nevertheless may be different in their technical requirements and their ease of use. In this respect, it is interesting to mention that Schwartz initially defined distributions as continuous linear operators $\Cnt[\infty]_c\to\Cnt[\infty]$, with a regular distribution $f$ acting as the convolution operator $\phi\mapsto f\conv \phi$ \cite[p.\ 230]{Schwartz01}, and only later simplified it to the theory that is in use today.

\section{Preliminaries}
We use multi-index notation for partial derivatives of higher order, i.e., for $\alpha=(\alpha_1,\dots,\alpha_d)\in\N^d$, we denote $\partial^\alpha f = \partial_1^{\alpha_1}\cdots\partial_d^{\alpha_d} f$,  $|\alpha| := \alpha_1+\dots+\alpha_d$ and $x^\alpha= x_1^{\alpha_1}\cdots x_d^{\alpha_d}$.

For $k\in\N\cup\{\infty\}$, we denote by $\Cnt[k]=\Cnt[k](\R^d)$ the set of all maps $\R^d\to\R$ with continuous partial derivatives up to order $k$ and by $\Cnt[k]_c=\Cnt[k]_c(\R^d)$ the set of all such maps with compact support. We denote $\Cnt:= \Cnt[0]$.

We denote the convolution product of $f$, $g \in \Cnt$ by
\[(f\conv g)(x):= \int_{\R^d} f(x-y)g(y)\,dy\]
if this integral converges (e.g., if $f$ or $g$ has compact support).

We use A.~Robinson's framework of nonstandard analysis \cite{Robinson,SL} (see also \cite{Goldblatt} for a gentle introduction), but the theory can equally well be phrased in E.~Nelson's IST framework of nonstandard analysis. We will occasionally give references to texts in IST, if a result is not readily available in the literature in Robinson's framework. Although both frameworks are essentially equivalent in practice, the reader should be warned that notations in both frameworks are different.

We denote $x\approx 0$ for an infinitesimal element $x\in\ster\R^d$, and $x\approx y$ iff $x-y\approx 0$. We denote $\Fin(\ster\R^d)$ for the set of finite elements in $\ster\R^d$, i.e.\ those bounded in norm by some real number, and denote $\ster\R^d_\infty := \ster\R^d\setminus \Fin(\ster\R^d)$. The standard part $\st x$ of $x\in \Fin(\ster\R^d)$ is the unique element of $\R^d$ such that $x\approx\st x$. We also write $x\approxin A$ as an abbreviation for $x\approx y$, for some $y\in A$. Notice that $x\in\Fin(\ster\R^d)$ iff $x\approxin\R^d$.

We call $f$: $\ster\R^d\to \ster\R$ $S$-continuous at $x\in\ster\R^d$ if $f(y)\approx f(x)$, $\forall y\approx x$. If $f$ is $S$-continuous on $\R^d$ and $f(x)\approxin\R$, $\forall x\in \R^d$, then there exists a unique $g\in \Cnt$ such that $f(x)\approx g(x)$, $\forall x\approxin \R^d$, which is called the standard part $\st f$ of $f$ \cite[\S 4.3]{Robert}, \cite[App.\ 2]{Oberguggenberger88} (cf.\ \cite[Thm.\ 4.5.10]{Robinson}).

For a standard function $f$: $\R^d\to\R$, the map $\ster f$: $\ster\R^d\to\ster\R$ is an extension of $f$. As is customary, we will also write $f$ for this extension. With the same abuse of notation, we consider $\Cnt[k]\subseteq\ster{\Cnt[k]}$. Similarly, we write $\partial^\alpha:= \ster\partial^\alpha$ and $\int:=\ster\int$.

We call \(f\): \(\ster\R^d\to\ster\R\) \(\ast\)-compactly (resp.\ S-compactly) supported if
\[f(x) = 0, \ \forall x\in \ster\R^d\text{ with }\abs x\ge R\]
for some $R\in\ster\R$ (resp.\ some $R\in\R$). By overspill, an internal $f$ is S-compactly supported iff $f(x)=0$ for each $x\in\ster\R^d_\infty$. By transfer, $\ster{\Cnt[k]_c}$ consists of all \(\ast\)-compactly supported functions in $\ster{\Cnt[k]}$.

Notice that $f$: $\ster\R^d\to\ster\R$ is S-continuous on $\R^d$ iff $f$ is S-continuous on $\Fin(\ster\R^d)$. If $f$ is also S-compactly supported, then $f$ is S-continuous on $\ster\R^d$.

\section{Motivating problem}
We use the problem to find a solution of a linear PDE by convolution, once a fundamental solution is known, as a motivation to develop our theory. The problem can be formulated as the following theorem in distribution theory:

\begin{theorem}
Consider the $m$-th order linear constant coefficient PDO $P(\partial): = \sum_{\abs\alpha\le m} c_\alpha \partial^\alpha$ (with $\alpha\in\N^d$, $c_\alpha\in\R$). Let $E$ be a fundamental solution for $P(\partial)$, i.e., $P(\partial) E = \delta$. Let $f\in\Cnt[\infty]_c(\R^d)$. Then
$u:=E\conv f\in\Cnt[\infty](\R^d)$ solves $P(\partial) u= f$ (in the classical sense).
\end{theorem}
Once the distributional framework is available, the theorem follows easily. We will prove an equivalent nonstandard version of this theorem. This will allow us to introduce definitions in a way directly motivated to solve the problem.\\
In the sequel, $P(\partial)$ will be an operator defined as in Thm.~1.

\section{Delta functions}
Dirac describes his intuition about his delta function as follows \cite[\S 15]{Dirac}:
\begin{quote}
\emph{To get a picture of $\delta(x)$, take a function of the real variable $x$ which vanishes everywhere except inside a small domain, of length $\varepsilon$ say, surrounding the origin $x = 0$, and which is so large inside this domain that its integral over this domain is unity.} (\dots) \emph{Then in the limit $\varepsilon \to 0$ this function will go over into $\delta(x)$.}
\end{quote}
It was clear to Dirac that there isn't any function $\R^d\to\R$ satisfying the requirements of his idealized delta function \cite[\S 15]{Dirac}:
\begin{quote}
\emph{$\delta(x)$ is not a function of $x$ according to the usual definition of a function}
\end{quote}
but functions $\ster\R^d\to\ster\R$ can satisfy Dirac's intuition where we choose $\varepsilon$ infinitesimal. As we want to be able to differentiate delta functions, we seek such functions in the space \(\ster{\Cnt[\infty]}\).

\begin{definition}
Let \(\psi\in\test\) with \(\int_{\R^d}\psi = 1\). Let \(\rho\approx 0\), \(\rho >0\). Then we call the map \(\delta\): \(\ster\R^d\to\ster\R\): \(\delta(x):= \frac{1}{\rho^d}\psi\big(\frac{x}\rho\big)\) a (nonstandard) model delta function.\footnote{This name is reminiscent of the notion of a model delta net \cite[Def.\ 7.9]{Oberguggenberger}.}
\end{definition}
By transfer, $\delta\in\ster{\Cnt[\infty]}$ and $\delta(x) = 0$ for each $x\in\ster\R^d$ with $x\not\approx 0$.

We now show the basic property of a delta function:
\begin{theorem}\label{model-delta}\label{ns-delta}
Let \(\delta\) be a model delta function. If \(f\in\ster{\Cnt}\) and \(f\) is S-continuous in \(0\) (in particular, if \(f\in \Cnt\)), then
\[\int_{\ster\R^d} f\delta \approx f(0).\]
\end{theorem}
\begin{proof}
Let \(a \approx 0\), \(a>0\) with \(\delta(x) = 0\) for \(\abs x\ge a\). As \(\int_{\ster\R^d}\delta = 1\) and \(\int_{\ster\R^d} \abs\delta = \int_{\R^d} \abs\psi\in\R\), we have
\[\abs{f(0) - \int_{\ster\R^d} f\delta} = \abs{\int_{\ster\R^d} (f(0) - f)\delta}
\le \max_{\abs x\le a}\abs{f(0) - f(x)} \int_{\ster\R^d}\abs\delta \approx 0
\]
since \(\max_{\abs x\le a}\abs{f(0) - f(x)}\approx 0\) by the S-continuity of $f$ in $0$.
\end{proof}
More generally, we define:
\begin{definition}\label{df-delta-orde-0}
Let $\delta\in\ster{\Cnt[\infty]}$. Then we call $\delta$ a delta function of order $0$ if $\int f\delta\approx f(0)$ for each S-continuous, S-compactly supported $f\in\ster\Cnt$.
\end{definition}
Each model delta function thus is a delta function of order $0$.

\begin{lemma}\label{delta-orde-0-support}
If $\delta\in\ster{\Cnt[\infty]}$ is an S-compactly supported delta function of order $0$, then $\int f\delta\approx f(0)$ for each $f\in\ster\Cnt$ that is S-continuous on $\R^d$.
\end{lemma}
\begin{proof}
For some $R\in\R$, we have $\delta(x)=0$ whenever $|x|> R$. Let $\psi\in\Cnt_c$ and $\psi(x)=1$ for $|x|\le R$. Then
\[\int_{\ster\R^d} f\delta = \int_{\ster\R^d} f\psi \delta \approx (f\psi) (0) = f(0)\]
since $f\psi\in \ster\Cnt$ is S-continuous and S-compactly supported.
\end{proof}

In Def.~\ref{df-delta-orde-0} we only consider S-compactly supported $f$ because we also want to allow functions as in the next example as delta functions:
\begin{example}\label{delta-orde-0}
Let $\psi\in\Cnt[\infty]$ with \(\int_{\R^d}\psi = 1\) and \(\int_{\R^d}|\psi|<+\infty\). Let \(\rho\approx 0\), \(\rho >0\). Then \(\delta(x):= \frac{1}{\rho^d}\psi\big(\frac{x}\rho\big)\) is a delta function of order \(0\).
\end{example}
\begin{proof}
If $R\in\ster\R_\infty$, $R>0$, then $\int_{|x|> R} |\psi|\approx 0$. Thus there exists $a \approx 0$, $a>0$ with $\int_{|x|> a} |\delta|\approx 0$. If $f\in\ster\Cnt$ is S-continuous and S-compactly supported, then
\begin{align*}
\abs{f(0) - \int_{\ster\R^d} f\delta} &= \abs{\int_{\ster\R^d} (f(0) - f)\delta}
\le \int_{\abs{x}\le a} \abs{f(0) - f}\abs{\delta} + \int_{\abs{x}> a} \abs{f(0) - f}\abs{\delta}\\ &\le \max_{\abs x\le a}\abs{f(0) - f(x)} \int_{\ster\R^d}\abs\delta + \max_{x\in\ster\R^d}\abs{f(0) - f(x)} \int_{|x|>a}\abs\delta \approx 0
\end{align*}
as $|f(x)-f(0)|\approxin \R$ for each $x\approxin\R^d$ by the S-continuity of $f$ \cite[Thm.~4.5.8]{Robinson}.
\end{proof}

Just as infinitely many hyperreals ($\approx a$) correspond to the real number $a$, infinitely many nonstandard delta functions correspond to the $\delta$-distribution.

\section{Weak solutions of linear PDE}
\begin{definition}
A fundamental solution of order \(0\) for \(P(\partial)\) is a function \(E\in\ster{\Cnt[\infty]}\) with the property that \(P(\partial) E\) is a delta function of order \(0\).
\end{definition}
By means of a fundamental solution we will solve the equation \(P(\partial) u= f\) (where \(f\in\Cnt\) is given) up to infinitesimals.

\begin{lemma}\label{conv-met-delta-orde-0}
Let \(\delta\in\ster{\Cnt[\infty]}\) be a delta function of order \(0\) and let \(f\in\ster\Cnt\) be S-continuous on $\R^d$. If one of them is S-compactly supported, then \(f\conv \delta(x)\approx f(x)\), \(\forall x\approxin\R^d\).
\end{lemma}

\begin{proof}
For $x\approxin \R^d$,
\[f\conv\delta(x) = \int f(x-y)\delta(y)\,dy \approx f(x)\]
since \(y\mapsto f(x-y)\) is S-continuous on $\R^d$ (and S-compactly supported if \(f\) is so).
\end{proof}

\begin{theorem}\label{lemma-fund-opl}\leavevmode
If \(E\) is a fundamental solution of order \(0\) for \(P(\partial)\) and \(f\in\Cnt_c\), then \(u:=E\conv f\in\ster{\Cnt[\infty]}\)
satisfies
\begin{equation}\label{eq-zwakke-oplossing-Cnt}
P(\partial)u (x) \approx f(x),\quad\forall x\approxin\R^d.
\end{equation}
\end{theorem}
\begin{proof}
As \(f\in\Cnt_c\), \(f\conv g\) is well defined for each \(g\in\ster\Cnt\), and thus
\[P(\partial) (E\conv f) (x) = (P(\partial) E)\conv f(x) \approx f(x),\quad\forall x\approxin\R^d\]
by Lemma \ref{conv-met-delta-orde-0}.
\end{proof}

The following theorem connects eq.~\eqref{eq-zwakke-oplossing-Cnt} to the classical notion of weak solution of a PDE. It is a digression which is not necessary for our development of the theory. It illustrates that also the concept of weak solution is an indirect way to express something that is naturally expressed with infinitesimals (a \emph{replacement construction for infinitesimals}, according to Laugwitz).
\begin{theorem}\label{distrib-oplossing}\leavevmode Let \(f\in\Cnt\).
\begin{enumerate}
\item If \(u\in\ster{\Cnt[\infty]}\) satisfies \eqref{eq-zwakke-oplossing-Cnt} and $v:=\st u\in\Cnt$ exists (i.e., if $u$ is S-continuous on $\R^d$ and $u(x)\approxin\R$, $\forall x\in \R^d$), then $v$ is a weak solution of \(P(\partial) v = f\).
\item Conversely, if \(v\in\Cnt\) is a weak solution of \(P(\partial) v =f\), then there exists \(u\in\ster{\Cnt[\infty]}\) with \(\st u = v\) satisfying \eqref{eq-zwakke-oplossing-Cnt}.
\end{enumerate}
\end{theorem}

\begin{proof}
1. If \(P(\partial) = \sum_{\abs \alpha\le m} c_\alpha \partial^\alpha\), then we denote \(P(-\partial) = \sum_{\abs \alpha\le m} c_\alpha (-1)^{\abs\alpha} \partial^\alpha\). By definition, \(v\) is a weak solution if
\[\int v P(-\partial)\phi =\int f \phi,\quad\forall \phi\in\test.\]
For \(\phi\in\Cnt[m]_c\), we have
\[\int v P(-\partial)\phi\approx \int u P(-\partial)\phi = \int \bigl(P(\partial) u\bigr)\phi \approx \int f \phi\]
because \(u(x)\approx v(x)\), $\forall x\approxin\R^d$ and by eq.\ \eqref{eq-zwakke-oplossing-Cnt}.\\
2. Choose an S-compactly supported delta function \(\delta\) of order \(0\) and define \(u:=v\conv \delta\in\ster{\Cnt[\infty]}\). Then $u(x)\approx v(x)$, $\forall x\approxin\R^d$. By transfer, we have \(\int v P(-\partial)\phi=\int f\phi\) for every \(\phi\in\ster\test\), and thus, for $x\approxin\R^d$,
\begin{align*}
P(\partial)u(x)&= v\conv(P(\partial) \delta) (x)=\int v(y)\bigl(P(\partial)\delta\bigr)(x-y)\,dy\\
&= \int v(y)P(-\partial)\bigl(\delta(x-y)\bigr)\,dy =f\conv\delta(x)\approx f(x).
\end{align*}
\end{proof}

\section{Strong solutions of linear PDE}
If $u$ satisfies eq.\ \eqref{eq-zwakke-oplossing-Cnt} and $\st u$ is defined on the whole of $\R^d$, then $\st u$ is not necessarily a solution (in the classical sense) of the PDE: $\st u$ is in general not differentiable. We recall a well-known condition that guarantees this:
\begin{definition}\cite[5.1.2]{Robert}
Let $f$: $\ster\R\to\ster\R$. We call $f$ S-differentiable at $x\in\ster\R$ with S-derivative $f'_S(x)\in\R$ if $\frac{f(y)- f(x)}{y-x}\approx f'_S (x)$, $\forall y\approx x$ ($y\ne x$).
\end{definition}
Clearly, S-differentiability implies S-continuity.
\begin{theorem}\label{st-afgeleide}
Let $f$: $\ster\R\to\ster\R$ be internal and S-differentiable on $\R$ with $f(x)\approxin\R$ for each $x\in\R$. Then $(\st f)'(x) = f'_S(x)$ for each $x\in\R$.
\end{theorem}
\begin{proof}
Let $x\in\R$. By overspill, S-differentiability implies that for each $\varepsilon\in\R$, $\varepsilon>0$ there exists $\delta\in\R$, $\delta>0$ such that
\[\abs[\Big]{\frac{f(y) - f(x)}{y-x} - f'_S(x)}\le \varepsilon, \quad\forall y\in\ster\R, 0\ne |y-x|\le\delta.\]
Choosing in particular $y\in\R$, we also have $\abs[\Big]{\frac{(\st f)(y) - (\st f)(x)}{y-x} - f'_S(x)}\le \varepsilon$.
\end{proof}

\begin{lemma}\label{verwisselen-st-afg-dim-d}
Let \(f\in\ster{\Cnt[1]}\). If \(\st f\) and $\st(\partial_j f)$ exist $(j=1,\dots,d)$, then \(\st f\in\Cnt[1]\) with \(\partial_j(\st f) = \st (\partial_j f)\).
\end{lemma}
\begin{proof}
Let \(a\in\R^d\) and let $f_j$ be the partial function
\[f_j: x_j\mapsto f(a_1,\dots,a_{j-1},x_j,a_{j+1},\dots,a_d).\]
By the mean value theorem, $f_j$ is S-differentiable on $\R$ with $f_j'(a_j)\approx f_{j S}'(a_j)$. Applying Thm.~\ref{st-afgeleide} to $f_j$, we further have $f_{j S}'(a_j)=(\st f_j)'(a_j) = \partial_j(\st f)(a)$.
\end{proof}

Again by the mean value theorem, we have:

\begin{lemma}\label{eindige-interne-afgeleide-is-Scnt}
Let \(f\in\ster{\Cnt[1]}\). If \(\partial_j f(x)\approxin \R\) for each \(x\approxin\R^d\) and each $j$, then \(f\) is S-continuous on $\R^d$.
\end{lemma}

We can now inductively handle PDE of arbitrary order:
\begin{theorem}\label{struct-FinCk}
Let \(f\in\ster{\Cnt[k]}\). If \(\partial^\alpha f(x)\approxin\R\), \(\forall x\approxin \R^d\), \(\forall\alpha\in\N^d, |\alpha|\le k\), and $\partial^\alpha f$ is S-continuous on $\R^d$, $\forall\alpha\in\N^d, |\alpha|\le k$, then \(\st f\in\Cnt[k]\) and \(\partial^\alpha(\st f) = \st (\partial^\alpha f)\), \(\forall \alpha\in\N^d, |\alpha|\le k\).
\end{theorem}
\begin{proof}
By the assumptions, $\st (\partial^\alpha f)$ exists, \(\forall \alpha\in\N^d, |\alpha|\le k\). By Lemma \ref{verwisselen-st-afg-dim-d}, \(\partial_j(\st \partial^\alpha f) = \st (\partial_j\partial^\alpha f)\), $\forall j$, $|\alpha|<k$.
\end{proof}
Combining this with Lemma \ref{eindige-interne-afgeleide-is-Scnt}, we obtain:
\begin{corollary}(cf. \cite[Prop.~2.8]{Oberguggenberger88})\label{struct-Fin-ster-D}
Let \(f\in\ster{\Cnt[\infty]}\). If \(\partial^\alpha f(x)\approxin\R\), \(\forall x\approxin \R^d\), \(\forall\alpha\in\N^d\), then \(\st f\in\Cnt[\infty]\) and \(\partial^\alpha(\st f) = \st (\partial^\alpha f)\), \(\forall \alpha\in\N^d\).
\end{corollary}

The conditions on \(f\) in the previous results suggest the following notations:
\begin{definition}(cf.\ \cite[\S 10.4]{SL})
Let $k\in\N\cup\{\infty\}$. Let $f$, $g$ $\in\ster{\Cnt[k]}$. Then we denote
\begin{align*}
f\approx_{\Cnt[k]} g&\text{ if }\partial^\alpha f(x)\approx \partial^\alpha g(x), \ \forall x\approxin \R^d,\ \forall \alpha\in\N^d, |\alpha|\le k\\
f\approxin \Cnt[k] &\text{ if }f\approx_{\Cnt[k]} g, \text{ for some }g\in \Cnt[k]\\
f\approx_{\Cnt[k]_c} g&\text{ if }f\approx_{\Cnt[k]} g \text{ and }f(x)=g(x), \ \forall x\in\ster\R^d_\infty\\
f\approxin \Cnt[k]_c &\text{ if }f\approx_{\Cnt[k]_c} g, \text{ for some }g\in \Cnt[k]_c.
\end{align*}
\end{definition}
By Thm.~\ref{struct-FinCk}, $f\approxin\Cnt[k]$ iff $\partial^\alpha f(x)\approxin\R$, $\forall x\approxin \R^d$, $\forall\alpha\in\N^d$, $|\alpha|\le k$ and $\partial^\alpha f$ is S-continuous on $\R^d$, $\forall\alpha\in\N^d, |\alpha|\le k$.
By Cor.~\ref{struct-Fin-ster-D}, $f\approxin\Cnt[\infty]$ iff $\partial^\alpha f(x)\approxin\R$, $\forall x\approxin \R^d$, $\forall\alpha\in\N^d$. Furthermore, $f\approxin\Cnt[k]_c$ iff $f\approxin \Cnt[k]$ and $f$ is S-compactly supported.

\begin{theorem}\label{opl-op-inf-na-impliceert-opl}
If \(f\in\Cnt\) and \(u\approxin\Cnt[\infty]\) satisfies equation \eqref{eq-zwakke-oplossing-Cnt}, then \(v:=\st u\in\Cnt[\infty]\) satisfies \(P(\partial)v=f\). 
\end{theorem}
\begin{proof}
By Cor.\ \ref{struct-Fin-ster-D}, \(P(\partial)(\st u) = \st P(\partial) u = f\).
\end{proof}
Combining this with Thm.\ \ref{lemma-fund-opl}, we find a solution to the equation (for $f\in\Cnt_c$) by means of a fundamental solution \(E\) of order 0, provided that \(E\conv f\approxin \Cnt[\infty]\). We now introduce a natural sufficient condition that guarantees this:
\begin{definition}\cite[\S 10.4]{SL}
\[
D' :=\Big\{f\in \ster{\Cnt[\infty]}: \int_{\ster\R^d} f\phi \approxin\R,\ \forall \phi\approxin \Cnt[\infty]_c\Big\}.
\]
For \(f\), \(g\in \ster{\Cnt}\) we write
\[
f\approx_{D'} g\quad \iff \quad \int f\phi\approx \int g\phi, \ \forall \phi\approxin \Cnt[\infty]_c.
\]
\end{definition}

It is appealing that we can apply typical techniques in nonstandard analysis also in this setting. E.g., the following is an analogue of Robinson's sequential lemma \cite[15.2]{Goldblatt}:
\begin{lemma}[Sequential lemma for \(\mathcal C^k\)]\label{rijenlemma-Ck}(cf.\ \cite[App.\ 2]{Oberguggenberger88})
Let \(k\in\N\cup\{\infty\}\). Let \((f_n)_{n\in\ster\N}\) be an internal hypersequence in \(\ster{\Cnt[k]}\). If \(f_n\approx_{\Cnt[k]} 0\) for each \(n\in\N\), then there exists \(\omega\in\ster\N_\infty\) such that \(f_n\approx_{\Cnt[k]} 0\) for each \(n\le \omega\) (\(n\in\ster\N\)).
\end{lemma}

\begin{proof}
By overspill on (for $k<\infty$)
\[
\N\subseteq \{n\in\ster\N: (\forall x\in\ster\R^d, \abs x\le n) (\forall \alpha\in\ster\N^d,\abs\alpha\le k) (\abs{\partial^\alpha f_n(x)}\le 1/n)\}
\]
resp.\ (for $k=\infty$)
\[
\N\subseteq \{n\in\ster\N: (\forall x\in\ster\R^d, \abs x\le n) (\forall \alpha\in\ster\N^d,\abs\alpha\le n) (\abs{\partial^\alpha f_n(x)}\le 1/n)\}.
\]
\end{proof}

\begin{lemma}\label{equiv-def-C'}
Let \(f\in\ster{\Cnt[\infty]}\).
\begin{enumerate}
\item \(f\approxin \Cnt\) \(\implies\) \(f\in D'\).
\item \(f\approx_{\Cnt} 0\) \(\implies\) \(f\approx_{D'} 0\).
\item
\(\displaystyle f\in D' \iff \int_{\ster\R^d} f\phi \approx 0, \ \forall \phi\approx_{\Cnt[\infty]_c} 0\).
\item If \(f\in D'\) and \(\alpha\in\N^d\), then \(\partial^\alpha f\in D'\).
\item If \(f\approx_{D'}0\) and \(\alpha\in\N^d\), then \(\partial^\alpha f\approx_{D'} 0\).
\item Let $f\in D'$ (resp.\ $f\approx_{D'} 0$) be S-compactly supported.\\Then
$\displaystyle \int_{\ster\R^d}f\phi\approxin\R$ (resp.\ $\approx 0$), $\forall \phi \approxin \Cnt[\infty]$, and \(\displaystyle \int_{\ster\R^d} f\phi \approx 0, \ \forall \phi\approx_{\Cnt[\infty]} 0\).
\end{enumerate}
\end{lemma}
\begin{proof}
1--2. Let \(\phi\approxin{\Cnt_c}\). Then \(\phi(x)=0\) for \(\abs x\ge R\) (some \(R\in\R\)), and thus
\[\abs[\bigg]{\int f\phi}\le \max_{\abs x\le R}\abs f \max_{\abs x\le R} \abs\phi\int_{\abs x\le R}1.\]
3. \(\Rightarrow\): Let \(\phi\approx_{\Cnt[\infty]_c} 0\). Then also \(n\phi\approx_{\Cnt[\infty]} 0\) for each \(n\in\N\). By the sequential lemma, it follows that also \(\omega\phi\approx_{\Cnt[\infty]} 0\) for some \(\omega\in\ster\N_\infty\), and thus \(\omega\phi\approxin \Cnt[\infty]_c\). By the assumptions, it follows that \(\omega\int_{\ster\R^d} f\phi\approxin\R\).

\(\Leftarrow\): Let \(\phi\approxin\Cnt[\infty]_c\). Then \(\varepsilon\phi\approx_{\Cnt[\infty]_c} 0\) for every \(\varepsilon\approx 0\). By the assumptions, it follows that \(\varepsilon\int_{\ster\R} f\phi\approx 0\) for every \(\varepsilon\approx 0\). Then \(\abs[\big]{\int_{\ster\R} f\phi}< \frac1\varepsilon\) for every \(\varepsilon\approx 0\), and thus \(\int_{\ster\R} f\phi\approxin \R\).\\
4--5. By integration by parts.\\
6. As \(f(x)=0\) for \(|x|\ge R\) (some \(R\in\R\)), we have $\int_{\ster\R^d} f\phi = \int_{\ster\R^d} f (\phi\psi)$, with $\psi\in\Cnt[\infty]_c$ with $\psi(x)=1$ if $|x|\le R$.
\end{proof}

\begin{theorem}\label{conv-in-Ckc'}\label{conv-in-D'}
Let $f,g\in\ster{\Cnt[\infty]}$, with at least one of them S-compactly supported.
\begin{enumerate}
\item If \(f\approxin\Cnt[\infty]\) and \(g\approx_{D'} 0\), then \(f\conv g\approx_{\Cnt[\infty]} 0\).
\item If \(f\approx_{\Cnt[\infty]} 0\) and \(g\in D'\), then \(f\conv g\approx_{\Cnt[\infty]} 0\).
\item If \(f\approxin\Cnt[\infty]\) and \(g\in D'\), then \(f\conv g\approxin \Cnt[\infty]\).
\item If \(f,g\in D'\), then \(f\conv g\in D'\).
\item If \(f\in D'\) and \(g\approx_{D'} 0\), then \(f\conv g\approx_{D'} 0\).
\end{enumerate}
\end{theorem}
\begin{proof}
1--3. By transfer, \(f\conv g\in\ster{\Cnt[\infty]}\). Let \(f\approxin\Cnt[\infty]_c\) and \(g\approx_{D'} 0\). Let \(\alpha\in\N^d\) and \(x\approxin\R^d\). Since \(y\mapsto \partial^\alpha f(x-y)\approxin \Cnt[\infty]_c\), we have \(\partial^\alpha (f\conv g)(x)=((\partial^\alpha f)\conv g)(x)\approx 0\). The other cases are similar (using Lemma \ref{equiv-def-C'}).\\
4--5. Let \(f,g\in D'\), with $f$ S-compactly supported. If \(\phi\approxin \test\), then \((f\conv g)\conv \phi = f\conv (g\conv \phi)\approxin\Cnt[\infty]\), because \(g\conv\phi\approxin\Cnt[\infty]\). In particular, \(((f\conv g)\conv \phi)(0)\approxin\R\). The other cases are similar.
\end{proof}

\begin{theorem}\label{fund-opl-sterke-opl}
If \(E\in D'\) is a fundamental solution of order \(0\) for \(P(\partial)\) and \(f\in\Cnt[\infty]_c(\R^d)\), then \(v:=\st (E\conv f)\in\Cnt[\infty](\R^d)\) is a solution of \(P(\partial) v = f\).
\end{theorem}
\begin{proof}
By Thm.\ \ref{lemma-fund-opl}, \(E\conv f\) satisfies eq.\ \eqref{eq-zwakke-oplossing-Cnt}. By Thm.\ \ref{conv-in-D'}, \(E\conv f\approxin\Cnt[\infty]\), and thus the result follows by Thm.\ \ref{opl-op-inf-na-impliceert-opl}.
\end{proof}

\begin{definition}
Let \(\psi\in\Cnt[\infty]_c\) with \(\int\psi=1\). We call \(\psi_n(x):=n^d\psi(nx)\) for each \(n\in\ster\N\). Then we call \((\psi_n)_{n\in\ster\N}\) a model delta sequence.
\end{definition}
\begin{theorem}\label{D'-eltair}
Let \(k\in\N\cup\{\infty\}\). If \(f\approxin\Cnt[k]\) and \(f\approx_{D'} 0\), then \(f\approx_{\Cnt[k]}0\).
\end{theorem}
\begin{proof}
Case \(k=0\): let \((\psi_n)_n\) be a model delta sequence. Then \(f\conv\psi_n\approx_{\Cnt[\infty]} 0\) for each \(n\in\N\) (Thm.\ \ref{conv-in-Ckc'}). By the sequential lemma this also holds for some \(\omega\in\ster\N_\infty\). Since \(\psi_\omega\) is a compactly supported delta function of order \(0\), we have \(f\approx_{\Cnt} f\conv\psi_\omega\approx_{\Cnt[\infty]} 0\) by Lemma \ref{conv-met-delta-orde-0}.\\
General case: if $f\approxin\Cnt[k]$, then by definition $\partial^\alpha f\approxin \Cnt$, \(|\alpha|\le k\). Apply the case \(k=0\) to \(\partial^\alpha f\).
\end{proof}

\begin{corollary}
Let $f\in D'$. If $f\approx_{D'} g$ for some $g\in \Cnt$, then $g$ is unique. We call $g=: \st f$ the standard part of $f$.
If moreover $f\approxin\Cnt$, then $\st f$ coincides with the standard part as defined in the Preliminaries.
\end{corollary}

\begin{remark}
By Thm.\ \ref{D'-eltair}, the relation \(\approx\) `fits itself' to the space to which the functions belong. This allows us to keep the notations light and in most cases drop the indices $_{D'}$ and $_{\Cnt[k]}$ to the relation $\approx$ without introducing ambiguity (for clarity, we will not do this in this survey).
\end{remark}

\section{Delta functions of higher order}
We have not yet recovered the full flexibility of distribution theory, because in concrete problems in analysis, more general delta functions occur:

\begin{theorem}(cf.\ \cite[Prop.\ 11.2]{Hoskins})
\label{delta-functie-algemener}
Let \(\psi\in\Cnt(\R)\) and \(\int_{-\infty}^{\infty} \psi = 1\) (as an improper Riemann integral). Let \(\rho\approx 0\), \(\rho>0\) and \(\delta(x):=\frac{1}{\rho}\psi(\frac x\rho)\). Then \(\int f\delta \approx f(0)\) for every \(f\approxin\Cnt[1]_c(\R)\).
\end{theorem}
\begin{proof}
Consider the antiderivative \(H(x):= \int_{-\infty}^x \delta(t)\,dt\). As \(\phi(x):= \int_{-\infty}^x\psi(t)\,dt\in\Cnt(\R)\) with \(\lim_{x\to\infty} \phi(x) =1\) and \(\lim_{x\to-\infty}\phi(x)=0\), we have \(\abs{H(x)}=\abs{\phi(x/\rho)}\approxin\R\), \(\forall x\in\ster\R\) and \(H(x)\approx 0\) if \(x\le -R\rho\) and \(H(x)\approx 1\) if \(x\ge R\rho\), for \(R\in\ster\R_\infty\), \(R>0\). Now choose in particular \(R\in\ster\R_\infty\) such that \(R\rho\approx 0\) (e.g., \(R:= 1/\sqrt{\rho}\)). If  \(f(x)=0\) for \(|x|\ge M\) (\(M\in\R\)), then
\[\int\limits_{\ster\R} f\delta = -\int\limits_{-M}^M f' H = -\underbrace{\int\limits_{-M}^{-R\rho} f' H}_{\approx 0} -\underbrace{\int\limits_{-R\rho}^{R\rho} f' H}_{\approx 0} - \int\limits_{R\rho}^M f' H\approx -\int\limits_{R\rho}^M f'= f(R\rho)\approx f(0)\]
since
\[
\abs{\int_{R\rho}^M f' H - \int_{R\rho}^M f'}\le M\max_{R\rho\le x\le M}\abs{H(x)-1}\max_{0\le x\le M}\abs{f'}\approx 0.
\]
\end{proof}

\begin{example}\label{Dirichlet-kern} (cf.\ \cite[\S 5.1]{Laugwitz89})
Let \(\psi(x):= \frac{\sin x}{\pi x}\). For \(\lambda\in\ster\R_\infty\), \(\lambda>0\), \(\delta(x):= \lambda \psi(\lambda x) = \frac{\sin \lambda x}{\pi x}\) is a delta function in the sense of Thm.\ \ref{delta-functie-algemener}. Thus \((f\conv \delta) (x)\approx f(x)\) for each \(f\approxin\Cnt[1]_c(\R)\) and for each \(x\approxin\R\). As
\[
\frac{1}{2\pi}\int_{-\lambda}^\lambda e^{i\omega(x-t)}d\omega = \frac{\sin\lambda(x-t)}{\pi(x-t)}=\delta(x-t)
\]
we obtain the following Fourier inversion formula for $f\in\Cnt[1]_c(\R)$, denoting the Fourier transform \(\fourier f(\omega):= \frac{1}{\sqrt{2\pi}}\int_\R f(t)e^{-i\omega t}\,dt\):
\[
f(x) \approx \int_{\ster\R} f(t)\delta(x-t)\,dt = \frac1{2\pi}\int_{\ster\R}\int_{-\lambda}^\lambda f(t)e^{i\omega(x- t)}\,dt\,d\omega
=\frac1{\sqrt{2\pi}}\int_{-\lambda}^\lambda \fourier f(\omega)e^{i\omega x}\,d\omega
\]
for any $x\in\R$ and \(\lambda\in\ster\R_\infty\), \(\lambda>0\). Thus \(f(x)=\frac1{\sqrt{2\pi}}\lim_{\lambda\to\infty}\int_{-\lambda}^\lambda \fourier f(\omega)e^{i\omega x}\,d\omega\).\\(The condition on the support of $f$ can of course be weakened.)
\end{example}

We thus define:
\begin{definition}
If $\delta\in\ster{\Cnt[\infty]}$ and $\int f \delta\approx f(0)$ for each $f\approxin\Cnt[\infty]_c$, then we call $\delta$ a (nonstandard) delta function.\\
$E\in\ster{\Cnt[\infty]}$ is a fundamental solution for $P(\partial)$ if $P(\partial) E$ is a delta function.\\
$u\in\ster{\Cnt[\infty]}$ is a weak solution of $P(\partial) u = f$ if $P(\partial) u\approx_{D'} f$. 
\end{definition}

As an extension of Lemma \ref{conv-met-delta-orde-0}, we have:
\begin{lemma}\label{delta-conv-algemeen}\label{delta-conv}
If \(\delta\) is a delta function and \(f\in D'\), and one of them is S-compactly supported, then \(f\conv \delta\approx_{D'} f\).
\end{lemma}
\begin{proof}
Assume first \(f\approxin\Cnt[\infty]\). As in the proof of Lemmas \ref{delta-orde-0-support} and \ref{conv-met-delta-orde-0}, \(f\conv\delta \approx_{\Cnt[\infty]} f\).\\
Now let \(f\in D'\) be arbitrary. Let \(\phi\approxin\Cnt[\infty]_c\). As \(\delta\conv\phi\approx_{\Cnt[\infty]} \phi\), we have \((f\conv\delta)\conv\phi=f\conv(\delta\conv\phi)\approx_{\Cnt[\infty]} f\conv\phi\) by Thm.\ \ref{conv-in-D'}. In particular, \((f\conv\delta)\conv\phi(0)\approx f\conv\phi(0)\).
\end{proof}

\begin{corollary}\label{gevolg-fund-opl}
If \(E\) is a fundamental solution for \(P(\partial)\) and \(f\in D'\) is S-compactly supported, then \(E\conv f\) is a weak solution for \(P(\partial) u=f\).
\end{corollary}
\begin{proof}
Similar to the proof of Thm.\ \ref{lemma-fund-opl} (now by Lemma \ref{delta-conv-algemeen} instead of Lemma \ref{conv-met-delta-orde-0}).
\end{proof}

\begin{theorem}\label{main-thm}
If \(P(\partial)\) is a PDO with fundamental solution \(E\in D'\) and \(f\in\Cnt[\infty]_c\), then \(v:=\st (E\conv f)\in\Cnt[\infty]\) is a solution of \(P(\partial) v = f\).
\end{theorem}
\begin{proof}
Similar to the proof of Thm.\ \ref{fund-opl-sterke-opl}.
\end{proof}

\section{Concluding remarks}
We can now work within this framework in the same way as in distribution theory: any operation on $D'$ which is independent of the representative modulo $\approx_{D'}$ gives rise to a distributional operation. Notice that in the nonstandard framework we are not constrained to distributional operations: all operations defined on functions in $\Cnt[\infty]$ are also defined on $\ster{\Cnt[\infty]}$---but the result of the operation may leave the space $D'$ or may depend on the representative modulo $\approx_{D'}$.

Although one can completely develop distribution theory in $D'$, it eventually becomes desirable to have the delta function (and other distributions) represented by one object (instead of infinitely many nonstandard functions which are $\approx_{D'}$-close to each other). To obtain this, one can simply consider equivalence classes modulo $\approx_{D'}$. The resulting space $D'/_{\approx_{D'}}$ is isomorphic to Schwartz's space ${\mathcal D}'$ \cite[\S 2]{Oberguggenberger88}. The situation is similar to developing elementary calculus in $\ster\Q$: it goes a long way (cf.\ \cite[Ch.\ 2]{SL}), but eventually, it becomes desirable to attach a single symbol to an irrational real number, and to pass to the quotient space $\Fin(\ster\Q)/_{\approx}$ which is isomorphic to $\R$.


\begin{thebibliography}{6}
\bibitem{Dirac} P.A.M.~Dirac. The Principles of Quantum Mechanics. 4th edition. Oxford University Press, 1958.

\bibitem{Goldblatt} R.~Goldblatt. Lectures on the hyperreals: an introduction to nonstandard analysis. Graduate Texts in Math., vol.\ 188, Springer, New York, 1998.

\bibitem{Hoskins-Pinto} R.F.~Hoskins, J.~Sousa Pinto. Theories of Generalised Functions. Woodhead Publishing, 2010.

\bibitem{Hoskins} R.F.~Hoskins. Delta functions. Introduction to generalised functions. 2nd edition. Woodhead Publishing, 2009.

\bibitem{Laugwitz86} D.~Laugwitz. Zahlen und Kontinuum. Eine Einf\"uhrung in die Infinitesimalmathematik.  Lehrb\"ucher und Monographien zur Didaktik der Mathematik, vol.\ 5. Bibliographisches Institut, 1986.

\bibitem{Laugwitz89} D.~Laugwitz, Definite Values of Infinite Sums: Aspects of the Foundations of Infinitesimal Analysis around 1820, Archive for History of Exact Sciences, 39 (3): 195--245 (1989).

\bibitem{Oberguggenberger88} M.~Oberguggenberger, Products of distributions: nonstandard methods. Zeitschrift Anal.\ Anw.\ 7(4): 347--365 (1988).

\bibitem{Oberguggenberger} M.~Oberguggenberger. Multiplication of distributions and applications to partial differential equations. Longman, 1992.

\bibitem{Robert} A.~Robert. Nonstandard Analysis. John Wiley \& Sons, 1988.

\bibitem{Robinson} A.~Robinson. Non-Standard Analysis. North-Holland, 1966.

\bibitem{Schwartz} L.~Schwartz. Th\'eorie des distributions. Hermann, Paris, 1966.

\bibitem{Schwartz01} L.~Schwartz. A mathematician grappling with his century. Birkh\"auser, 2001.

\bibitem{SL} K.D.~Stroyan and W.A.J.~Luxemburg. Introduction to the theory of infinitesimals. Pure and Applied Math.\ Series. Academic Press, 1976.

\bibitem{Vernaeve} H.~Vernaeve, The local structure of nonstandard representatives of distributions, Portugaliae Math.\ (2008) 65: 321--337. 
\end{thebibliography}
\end{document}